\documentclass[a4paper, 10pt]{amsart}

\usepackage{arydshln}
\usepackage[latin1]{inputenc}
\usepackage{enumitem}
\usepackage{nameref}
\usepackage{hyperref}
\usepackage[nospace,noadjust]{cite}
\usepackage{lastpage, setspace}
\usepackage{graphicx, pst-all}
\usepackage[all]{xy}
\usepackage{amsmath, amsthm, amssymb, tensor}
\usepackage{verbatim}
\usepackage{amsmath,bm}

\SelectTips{cm}{}
\newdir{ >}{{}*!/-5pt/\dir{>}}

\DeclareMathOperator\Aut{Aut}
\DeclareMathOperator\Sym{Sym}
\DeclareMathOperator\id{id}

\DeclareMathOperator\bbN{\mathbb{N}}

\DeclareMathOperator\calF{\mathcal{F}}
\DeclareMathOperator\calH{\mathcal{H}}
\DeclareMathOperator\calP{\mathcal{P}}

\theoremstyle{remark}
\newtheorem{theorem}{Theorem}
\newtheorem{lemma}[theorem]{Lemma}

\newtheorem{definition}[theorem]{Definition}
\newtheorem{remark}[theorem]{Remark}
\newtheorem{proposition}[theorem]{Proposition}

\makeatletter
\newtheorem*{rep@theorem}{\rep@title}
\newcommand{\newreptheorem}[2]{%
\newenvironment{rep#1}[1]{%
 \def\rep@title{#2 \scshape \ref{##1}}%
 \begin{rep@theorem}}%
 {\end{rep@theorem}}}
\makeatother

\newreptheorem{theorem}{Theorem}
\newreptheorem{proposition}{Proposition}

\title{Prime Localizations of Burger-Mozes type groups}
\author{Stephan Tornier}
\date{\today}

\begin{document}

\begin{abstract}
This article concerns Burger--Mozes universal groups acting on regular trees locally like a given permutation group of finite degree. We also consider locally isomorphic generalizations of the former due to Le Boudec and Lederle. For a large class of such permutation groups and primes $p$ we determine their local $p$-Sylow subgroups as well as subgroups of their $p$-localization, which is identified as a group of the same type in certain cases.
\end{abstract}

\maketitle
\vspace{-0.5cm}

\section{Introduction}

The concept of prime localization of a totally disconnected locally compact group $G$ was introduced by Reid in \cite{Rei13}. Let $p$ be prime. A \emph{local $p$-Sylow subgroup} of $G$ is a maximal pro-$p$ subgroup of a compact open subgroup of $G$. The \emph{$p$-localization} $G_{(p)}$ of $G$ is defined as the commensurator $\mathrm{Comm}_{G}(S)$ of a local $p$-Sylow subgroup $S$ of $G$, equipped with the unique group topology which makes the inclusion of $S$ into $G_{(p)}=\mathrm{Comm}_{G}(S)$ continuous and open. We refer the reader to \cite{Rei13} for general properties of prime localization and its applications, of which we highlight the scale function introduced by Willis in \cite{Wil94}.

For a set $\Omega$ of cardinality $d\in\bbN_{\ge 3}$ and $\smash{F\le F'\le\Sym(\Omega)}$, the Burger--Mozes group $\mathrm{U}(F)$ and the Le Boudec group $\mathrm{G}(F,F')$ act on a suitably colored $d$-regular tree $T_{d}$ with local action prescribed by $F$ and $F'$. Lederle's coloured Neretin group $\mathrm{N}(F)$ consists of almost automorphisms of $T_{d}$ associated to $\mathrm{U}(F)$.

For a large family of the above groups, we determine local $p$-Sylow subgroups in terms of a $p$-Sylow subgroup of $F$. Let $T\subseteq T_{d}$ denote a finite subtree. For $H\le\Aut(T_{d})$ we let $H_{T}$ denote the pointwise stabilizer of $T$ in $H$.

\begin{repproposition}{prop:local_sylow_uf}
Let $F\!\le\!\Sym(\Omega)$ and $F(p)\le F$ a $p$-Sylow subgroup. Then $\mathrm{U}(F(p))_{T}$ is a $p$-Sylow subgroup of $\mathrm{U}(F)_{T}$ if and only if so is $F(p)_{\omega}\le F_{\omega}$ for all $\omega\in\Omega$.
\end{repproposition}

After collecting criteria and examples for the above situation we determine general subgroups of their $p$-localization which we use to identify the latter as a group of the same type in certain cases. Recall that $\mathrm{U}(F)=\mathrm{G}(F,F)$. In the following, $\smash{\widehat{F}}$ denotes the maximal subgroup of $\Sym(\Omega)$ preserving the partition $F\backslash\Omega$ setwise.

\begin{reptheorem}{thm:localization_uf}
Let $\smash{F\le F'\le\widehat{F}\le\Sym(\Omega)}$ and $F(p)\le F$ a $p$-Sylow subgroup of~$F$. Assume that we have $F\backslash\Omega=F(p)\backslash\Omega$ and $N_{F_{\omega}'}(F(p)_{\omega})=F(p)_{\omega}$ for all $\omega\in\Omega$. Then~$\mathrm{G}(F,F')_{(p)}=\mathrm{G}(F(p),F')$.\end{reptheorem}

\begin{reptheorem}{thm:localization_nf}
Let $F\le\Sym(\Omega)$ and $F(p)\le F$ a $p$-Sylow subgroup. If~$F\backslash\Omega=F(p)\backslash\Omega$ and $N_{\widehat{F}_{\omega}}(F(p)_{\omega})=F(p)_{\omega}$ for all $\omega\in\Omega$ then $\mathrm{N}(F)_{(p)}=\mathrm{N}(F(p))$.
\end{reptheorem}

\subsubsection*{Acknowledgements} The author would like to thank The University of Newcastle at which this research was carried out for its hospitality and its group theory seminar, in particular Colin Reid, for many helpful discussions. The author is particularly grateful to George Willis for his ongoing support, initiation of the seminar and financial aid through the ARC grant DP120100996. Finally, the author benefitted from partial support through the SNSF Doc.Mobility fellowship 172120, and appreciates a referee's comments.

\section{Preliminaries}\label{sec:introduction} In order to provide concise definitions of Burger--Mozes type groups, we adopt Serre's graph theory notation, see \cite{Ser03}: A \emph{graph} consists of a \emph{vertex set} $V$ and an \emph{edge set} $E$, together with a fixed-point-free involution of $E$ denoted by $e\mapsto\overline{e}$ and maps $o,t:E\to V$, providing the \emph{origin} and \emph{terminus} of an edge $e$ such that $o(\overline{e})=t(e)$ and $t(\overline{e})=o(e)$. For $x\in V$, let $E(x):=\{e\in E\mid o(e)=x\}$ be the set of edges issuing from $x$. Let $\Omega$ be a set of cardinality $d\in\bbN_{\ge 3}$ and $T_{d}=(V,E)$ the $d$-regular tree. A \emph{legal coloring} $l$ of $T_{d}$ is a map $l:E\to\Omega$ such that for every $x\in V$ the map $l_{x}:E(x)\to\Omega,\ y\mapsto l(y)$ is a bijection, and $l(e)=l(\overline{e})$ for all $e\in E$. Given $T\subseteq T_{d}$ and $H\le\Aut(T_{d})$, we let $H_{T}$ denote the fixator of $T$ in $H$.

\subsubsection*{Burger--Mozes Groups} The universal groups introduced by Burger--Mozes in \cite[Section 3.2]{BM00} provide an equally rich and manageable class of groups acting on trees. The map $\sigma:\Aut(T_{d})\times V\to\Sym(\Omega),\ (g,x)\mapsto l_{g x}\circ g\circ l_{x}^{-1}$ captures the local permutation $\sigma(g,x)$ of an automorphism $g\in\Aut(T_{d})$ at $x\in V$. To every permutation group $F\le\Sym(\Omega)$ we associate a universal group acting on $T_{d}$ locally like $F$.

\begin{definition}
Let $F\le\Sym(\Omega)$. Set $\mathrm{U}(F):=\{g\in\Aut(T_{d})\mid \forall x\in V:\ \sigma(g,x)\in F\}$.
\end{definition}

Passing to a different legal coloring amounts to passing to a conjugate of $\mathrm{U}(F)$ in $\Aut(T_{d})$ which justifies omitting explicit reference to the legal coloring. For example, $\mathrm{U}(\Sym(\Omega))=\Aut(T_{d})$ whereas $\mathrm{U}(\{\id\})$ is the discrete cocompact subgroup generated by the color-preserving inversions of the edges in $E(x)$ for a given $x\in V$.

\begin{remark}\label{rem:uf_elements}
Let $F\le\Sym(\Omega)$. Elements of $\mathrm{U}(F)$ are readily constructed: Given $v,w\in V(T_{d})$ and $\tau\in F$, define $g:B(v,1)\to B(w,1)$ by setting $g(v)=w$ and $\sigma(g,v)=\tau$. Given a collection or permutations $(\tau_{\omega})_{\omega\in\Omega}$ such that $\tau(\omega)=\tau_{\omega}(\omega)$ for all $\omega\in\Omega$ there is a unique extension of $g$ to $B(v,2)$ such that $\sigma(g,v_{\omega})=\tau_{\omega}$ where $v_{\omega}\in S(v,1)$ is the unique vertex with $l(v,v_{\omega})=\omega$. Then proceed iteratively.
\end{remark}

Recall that $\Aut(T_{d})$ is a totally disconnected locally compact group with the permutation topology for its action on $V$. The following collection of properties is implicit in \cite[Section 3.2]{BM00} and elaborated in \cite[Proposition 4.6]{GGT16}.

\begin{proposition}[{\cite[Section 3.2]{BM00}}]\label{prop:uf_properties}
Let $F\le\Sym(\Omega)$. Then the group $\mathrm{U}(F)$ is
\begin{itemize}
 \item[(i)] closed in $\mathrm{Aut}(T_{d})$,
 \item[(ii)] locally permutation isomorphic to $F$,
 \item[(iii)] vertex-transitive,
 \item[(iv)] edge-transitive if and only if $F$ is transitive, and
 \item[(v)] discrete in $\Aut(T_{d})$ if and only if $F$ is semiregular.
\end{itemize}
\end{proposition}

As a consequence of the above, $\mathrm{U}(F)$ is a (compactly generated) totally disconnected locally compact group in its own right for the subspace topology of $\Aut(T_{d})$. For future reference, we also state the following, see \cite[Section 4.1]{GGT16}.

\begin{proposition}\label{prop:uf_tits}
Let $F\le\Sym(\Omega)$. Then $\mathrm{U}(F)$ satisfies Tits' Independence Property.
\end{proposition}

\subsubsection*{Le Boudec Groups} In \cite{Bou16}, Le Boudec introduces groups acting on $T_{d}$ locally like $F\le\Sym(\Omega)$ \emph{almost} everywhere. The precise definition reads as follows.

\begin{definition}
Let $F\le\Sym(\Omega)$. Set
\begin{displaymath}
 \mathrm{G}(F):=\{g\in\Aut(T_{d})\mid \text{the set $\{x\in V\mid \sigma(g,x)\not\in F\}$ is finite}\}.
\end{displaymath}
\end{definition}

Notice that $\mathrm{U}(F)$ is a subgroup of $\mathrm{G}(F)$. We equip $\mathrm{G}(F)$ with the unique group topology making the inclusion $\mathrm{U}(F)\rightarrowtail\mathrm{G}(F)$ continous and open. It exists essentially due to the fact that $\mathrm{G}(F)$ commensurates a compact open subgroup of $\mathrm{U}(F)$, see \cite[Lemma 3.2]{Bou16}. We state explicitly that this topology differs from the subspace topology of $\Aut(T_{d})$, see e.g. Proposition \ref{prop:gf_res_disc} below. However, it entails that $\mathrm{G}(F)$ is locally isomorphic to $\mathrm{U}(F)$.

Given $g\in\mathrm{G}(F)$, a vertex $v\in V$ with $\sigma(g,v)\not\in F$ is a \emph{singularity}. The local action at singularities is restricted as follows.

\begin{lemma}[{\cite[Lemma 3.3]{Bou16}}]
Let $F\le\Sym(\Omega)$ and $g\in\mathrm{G}(F)$ with a singularity $v\in V$. Then $\sigma(g,v)$ preserves the partition $F\backslash\Omega$ of $\Omega$ into $F$-orbits setwise.
\end{lemma}

For $F\le\Sym(\Omega)$, the maximal subgroup of $\Sym(\Omega)$ which preserves the partition $F\backslash\Omega=\bigsqcup_{i\in I}\Omega_{i}$ setwise is $\smash{\widehat{F}}:=\prod_{i\in I}\Sym(\Omega_{i})$.

\begin{definition}
Let $\smash{F\le F'\le\widehat{F}\le\Sym(\Omega)}$. Set $\mathrm{G}(F,F'):=\mathrm{G}(F)\cap\mathrm{U}(F')$.
\end{definition}

We remark that $\mathrm{G}(F,F)=\mathrm{U}(F)$ and $\smash{\mathrm{G}(F,\widehat{F})=\mathrm{G}(F)}$.

\begin{proposition}\label{prop:gf_res_disc}
Let $F\smash{\lneq F'\le\widehat{F}\le\Sym(\Omega)}$ and $b\in V(T_{d})$. Then $\mathrm{G}(F,F')_{b}$ is non-compact and residually discrete.
\end{proposition}

\begin{proof}
The vertex stabilizer $\mathrm{G}(F,F')_{b}$ can be written as the (strictly) increasing union $\mathrm{G}(F,F')_{b}=\bigcup_{n\in\bbN} K_{n}$ of the open sets $K_{n}$, consisting of the elements of $\mathrm{G}(F,F')_{b}$ whose singularities are contained in $B(b,n)$. Hence it is non-compact.

As to residual discreteness, an identity neighbourhood basis of $\mathrm{G}(F,F')_{b}$ consisting of open normal subgroups is given by the collection $\big(\mathrm{G}(F,F')_{B(b,n)}\big)_{n\in\bbN}$.
\end{proof}

Le Boudec groups enjoy many interesting properties, see \cite[Introduction]{Bou16}.

\subsubsection*{Lederle Groups} As before, we consider the $d$-regular tree $T_{d}=(V,E)$ equipped with a legal coloring and a base vertex $b\in V$. Further, let $F\le\Sym(\Omega)$. In \cite{Led17}, Lederle introduces an intriguing, locally isomorphic version of $\mathrm{U}(F)$ resembling Neretin's group \cite{Ner03} and thereby generalizes Neretin's construction.

Towards a precise definition, we recall the following from \cite[Section 3.2]{Led17}. A finite subtree $T\subseteq T_{d}$ is \emph{complete} if it contains $b$ and all its non-leaf vertices have valency $d$. We denote the set of leaves of $T$ by $L(T)\subseteq V(T_{d})$. Given a leaf $v\in L(T)$, let $T_{v}$ denote the subtree of $T_{d}$ spanned by $v$ and those vertices outside $T$ whose closest vertex in $T$ is $v$. Then define $T_{d}\backslash T:=\bigsqcup_{v\in L(T)}T_{v}$, a forest of $|L(T)|$ trees.

Let $H\!\le\!\Aut(T_{d})$. Given finite complete subtrees $T,T'\subseteq T_{d}$ with $|L(T)|\!=\!|L(T')|$, a forest isomorphism $\varphi:T_{d}\backslash T\to T_{d}\backslash T'$ such that for every $v\in L(T)$ there is $h_{v}\!\in\! H$ with $\varphi|_{T_{v}}=h_{v}|_{T_{v}}$ is an \emph{$H$-honest almost automorphism of $T_{d}$}. Two $H$-honest almost automorphisms of $T_{d}$ given by $\varphi:T_{d}\backslash T_{1}\to T_{d}\backslash T_{1}'$ and $\psi:T_{d}\backslash T_{2}\to T_{d}\backslash T_{2}'$ are \emph{equivalent} if there exists a finite complete subtree $T\supseteq T_{1}\cup T_{2}$ with $\varphi|_{T_{d}\backslash T}=\psi|_{T_{d}\backslash T}$. Notice that for any finite complete subtree $T\supseteq T_{1}$ there is a unique finite complete subtree $T'\supseteq T_{1}'$ and representative $\varphi':T_{d}\backslash T\to T_{d}\backslash T'$ of $\varphi$; analogously for $T_{1}'$. Hence we may pick a finite complete subtree $T\supseteq T_{1}'\cup T_{2}$ and representatives of $\varphi$ and $\psi$ with codomain and domain equal to $T_{d}\backslash T$ respectively, thus allowing for a composition of equivalence classes of $H$-honest almost automorphisms. Lederle's coloured Neretin groups (original notation $\calF(\mathrm{U}(F))$) can now be defined as follows.

\begin{definition}
Let $F\le\Sym(\Omega)$. Set
\begin{displaymath}
 \mathrm{N}(F):=\{[\varphi]\mid\varphi \text{ is a $\mathrm{U}(F)$-honest almost autormorphism of $T_{d}$}\}.
\end{displaymath}

\end{definition}

Observe that $\mathrm{N}(F)\cap\Aut(T_{d})=\mathrm{G}(F)$. As before, there exists a unique group topology on $\mathrm{N}(F)$ such that the inclusion $\mathrm{U}(F)\rightarrowtail\mathrm{N}(F)$ is open and continuous. This is essentially due to the fact that $\mathrm{N}(F)$ commensurates a compact open subgroup of $\mathrm{U}(F)$, see \cite[Proposition 2.24]{Led17}. Overall, given $F\le\Sym(\Omega)$ we have the following continuous and open injections:
\begin{displaymath}
 \xymatrix{
  \mathrm{U}(F) \ar@{ >->}[r] & \mathrm{G}(F) \ar@{ >->}[r] & \mathrm{N}(F).
 }
\end{displaymath}

\section{Local Sylow Subgroups} This section is concerned with determining local Sylow subgroups of the Burger--Mozes type groups. Throughout, $\Omega$ denotes a set of cardinality $d\in\bbN_{\ge 3}$ and $p$ is a prime. We consider the $d$-regular tree $T_{d}=(V,E)$ with a fixed legal coloring and base vertex $b\in V$. Furthermore, $T$ denotes a finite subtree of $T_{d}$.

\vspace{0.2cm}
Note that it suffices to consider $\mathrm{U}(F)$: Any local Sylow subgroup of $\mathrm{U}(F)$ is also a local Sylow subgroup of $\mathrm{G}(F,F')$ and $\mathrm{N}(F)$ by definition of the topologies.

In a sense, the following proposition provides local $p$-Sylow subgroups of $\mathrm{U}(F)$ in the case where the operations of taking a $p$-Sylow subgroup and taking point stabilizers commute for $F$. It is the basis of all subsequent statements about the $p$-localization of Burger--Mozes type groups and amends \cite[Lemma 4.2]{Rei13}.

\begin{proposition}\label{prop:local_sylow_uf}
Let $F\le\Sym(\Omega)$ and $F(p)\le F$ a $p$-Sylow subgroup. Then $\mathrm{U}(F(p))_{T}$ is a $p$-Sylow subgroup of $\mathrm{U}(F)_{T}$ if and only if so is $F(p)_{\omega}\le F_{\omega}$ for all $\omega\in\Omega$.
\end{proposition}

\begin{proof}
First, assume that $T$ consists of a single vertex $b\in V$. The sphere $S(b,k)\subset V$ of radius $k$ around $b\in V$ is, via the given legal coloring, in natural bijection with
\begin{displaymath}
 P_{k}:=\{w=(\omega_{1},\ldots,\omega_{k})\in\Omega^{k}\mid \forall i\in\{1,\ldots,k-1\}:\ \omega_{i+1}\neq \omega_{i}\}.
\end{displaymath}
The restriction of $\mathrm{U}(F)$ to $S(b,k)$ yields a subgroup of $\Sym(S(b,k))$ of cardinality given by $\left\vert\mathrm{U}(F)_{b}|_{S(b,1)}\right\vert=|F|$ and $\left\vert\mathrm{U}(F)_{b}|_{S(b,k+1)}\right\vert=\left\vert\mathrm{U}(F)_{b}|_{S(b,k)}\right\vert\cdot\prod_{w\in P_{k}}|F_{\omega_{k}}|$. The maximal powers of $p$ dividing $\left\vert\mathrm{U}(F)_{b}|_{S(b,k)}\right\vert$ and $\left\vert\mathrm{U}(F(p))_{b}|_{S(b,k)}\right\vert$ are hence equal for all $k\in\bbN_{0}$ if and only if $F(p)_{\omega}\le F_{\omega}$ is a $p$-Sylow subgroup for all $\omega\in\Omega$.

Similarly, when $T$ is not a single vertex, the size of the restriction of $\mathrm{U}(F)_{T}$ to a sufficiently larger subtree is a product of the $|F_{\omega}|$ involving \emph{all} $\omega\in\Omega$.
\end{proof}

For transitive $F\le\Sym(\Omega)$, it suffices to check the above criterion for one choice of a $p$-Sylow subgroup $F(p)$ of $F$ and all $\omega\in\Omega$. We now identify classes of permutation group and values of $p$ to which Proposition \ref{prop:local_sylow_uf} applies. For the symmetric and alternating groups we have the following, complete description.

\begin{proposition}
Let $F=\Sym(\Omega)$ or $F=\mathrm{Alt}(\Omega)$ and $F(p)\le F$ a $p$-Sylow subgroup. Further, let $p^{s}$ ($s\in\bbN_{0}$) be the maximal power of $p$ dividing $d$. Then $F(p)_{\omega}\le F_{\omega}$ is a $p$-Sylow subgroup for all $\omega\in\Omega$ if and only if either
\begin{itemize}
 \item[(i)] $p>d$, or
 \item[(ii)] $s\ge 1$ and $p^{s+1}>d$, or
 \item[(iii)] $F=\mathrm{Alt}(\Omega)$ and $(d,p)=(3,2)$.
\end{itemize}
\end{proposition}

\begin{proof}
If $p>d$ then $F(p)$ is trivial and so is any $p$-Sylow subgroup of $F_{\omega}$. Now assume $p\le d$ and consider the following diagram of subgroups of $F$ and indices.
\begin{minipage}{0.39\linewidth}
\vspace{-0.15cm}
\begin{displaymath}
 \xymatrix@C=0.7cm@R=0.01cm{
    & F \ar@{-}[dl]_{d} \ar@{-}[dr]^{k} & \\
    F_{\omega} \ar@{-}[dr] & & F(p) \ar@{-}[dl]^{p^{r_{\omega}}} \\
    & F(p)_{\omega} &
 }
\end{displaymath} 
\end{minipage}
\hfill
\begin{minipage}{0.59\linewidth}
For every $\omega\in\Omega$ we have $[F:F_{\omega}]=|F\cdot\omega|=d$ and $[F(p):F(p)_{\omega}]=|F(p)\cdot\omega|=p^{r_{\omega}}$ for some $r_{\omega}\in\bbN_{0}$. Note that $p\nmid k$ by definition. Now examine the equation $d\cdot[F_{\omega}:F(p)_{\omega}]=k\cdot p^{r_{\omega}}$.
\end{minipage}
If $F(p)$ is trivial then $F=\mathrm{Alt}(\Omega)$ and $p$ is even, hence (iii). Now assume that $F(p)$ is non-trivial. Then there is $\omega\in\Omega$ such that $r_{\omega}\ge 1$. Thus, if $p\nmid d$, then $p\mid[F_{\omega}:F(p)_{\omega}]$ and hence $F(p)_{\omega}$ is not a $p$-Sylow subgroup of $F_{\omega}$. We conclude that the condition $s\ge 1$ is necessary. Note that the biggest $p^{r_{\omega}}$ ($\omega\in\Omega$) which occurs is given by the biggest power of $p$ which is smaller than or equal to $d$ due to the iterated wreath product structure of $F(p)$. As $p\nmid k$ we conclude (ii).

Conversely, suppose $s\ge 1$ and $p^{s+1}\ge d$. If $p$ is odd, or $F=\Sym(\Omega)$ and $p$ is even, then $F(p)$ is a direct product of $s$-fold iterated wreath products and the maximum power of $p$ dividing $[F(p):F(p)_{\omega}]$ and $[F:F_{\omega}]$ is $p^{s}$ in both cases. The same index assertions hold for $F=\mathrm{Alt}(\Omega)$ and $p$ even.
\end{proof}

For a general permutation group $F\le\Sym(\Omega)$ and $\omega\in\Omega$ we have
\begin{displaymath}
 |F(p)\cdot\omega|=\frac{|F(p)|}{|F(p)_{\omega}|}=\frac{|F(p)|\cdot[F_{\omega}:F(p)_{\omega}]}{|F_{\omega}|}=\frac{[F_{\omega}:F(p)_{\omega}]}{[F:F(p)]}\cdot |F\cdot\omega|.
\end{displaymath}
by the orbit-stabilizer theorem. In particular, we conclude the following.

\begin{proposition}\label{prop:local_sylow_orbits}
Let $F\le\Sym(\Omega)$ and $F(p)\le F$ a $p$-Sylow subgroup. Assume that $F\backslash\Omega=F(p)\backslash\Omega$. Then $F(p)_{\omega}\le F_{\omega}$ is a $p$-Sylow subgroup for all $\omega\in\Omega$. \qed
\end{proposition}

\begin{proposition}
Let $|\Omega|=p^{n}$ and $F\le\Sym(\Omega)$ transitive. Also, let $F(p)\le F$ be a $p$-Sylow subgroup. Then so is $F(p)_{\omega}\le F_{\omega}$ for all $\omega\in\Omega$ and $F(p)$ is transitive.
\end{proposition}

\begin{proof}
In this case, the above equation reads
\begin{displaymath}
 |F(p)\cdot\omega|=\frac{[F_{\omega}:F(p)_{\omega}]}{[F:F(p)]}\cdot p^{n}.
\end{displaymath}
As always, $|F(p)\cdot\omega|$ is a power of $p$ and bounded by $|\Omega|=p^{n}$. Since $p$ does not divide $[F:F(p)]$ the above implies that $p$ does not divide $[F_{\omega}:F(p)_{\omega}]$.
\end{proof}

\section{Prime Localizations}\label{sec:localization}

This section is concerned with the $p$-localizations of Burger--Mozes type groups. Recall that for groups $H\le G$ one defines the \emph{commensurator of $H$ in $G$} by
\begin{displaymath}
 \mathrm{Comm}_{G}(H):=\{g\in G\mid [H:H\cap gHg^{-1}]<\infty \text{ and } [gHg^{-1}:gHg^{-1}\cap H]<\infty\}.
\end{displaymath}
The \emph{$p$-localization} of a totally disconnected locally compact group $G$ is defined as the commensurator $\mathrm{Comm}_{G}(S)$ of a local $p$-Sylow subgroup $S$ of $G$, equipped with the unique group topology that makes the inclusion of $S$ into $G_{(p)}:=\mathrm{Comm}_{G}(S)$ continuous and open. Then the inclusion $\mathrm{Comm}_{G}(S)\to G$ is continuous.

\vspace{0.2cm}
The following lemma due to Caprace--Monod \cite[Section 4]{CM11} and Caprace--Reid--Willis \cite[Corollary 7.4]{CRW13} is crucial for the subsequent statements of this section. See also \cite{Wes15}.

\begin{lemma}\label{lem:comm_normalizer}
Let $G$ be residually discrete, locally compact and totally disconnected. Further, let $K\le G$ be compact. Then $\mathrm{Comm}_{G}(K)=\bigcup_{L\le_{o} K}N_{G}(L)$.
\end{lemma}

\begin{proof}
Every element of $G$ which normalizes an open subgroup of $K$ commensurates $K$ because open subgroups of $K$ have finite index in $K$ given that $K$ is compact.

Conversely, let $g\in\mathrm{Comm}_{G}(K)$ and consider $H:=\langle K,g\rangle$. Then $H$ is a compactly generated open subgroup of $\mathrm{Comm}_{G}(K)$ and hence a compactly generated, totally disconnected locally compact group in its own right. It inherits residual discreteness from $\mathrm{Comm}_{G}(K)$ which injects continuously into the residually discrete group $G$. By \cite[Corollary 4.1]{CM11}, $H$ has an identity neighbourhood basis consisting of compact open normal subgroups. Hence $g$ normalizes an open subgroup of $K$.
\end{proof}

Now, let $F\le F'\le\widehat{F}\le\Sym(\Omega)$. In the case of Proposition \ref{prop:local_sylow_uf}, the following proposition identifes certain subsets of the $p$-localization of $\mathrm{G}(F,F')$ and thereby expands \cite[Lemma 4.2]{Rei13} given that $\mathrm{U}(F)=\mathrm{G}(F,F)$. We establish the following notation: Given partitions $\calP:=(P_{i})_{i\in I}$ of $V$ and $\calH=(H_{j})_{j\in J}$ of $H\le\Sym(\Omega)$, let
\begin{displaymath}
 \Gamma_{\calP}(\calH):=\{g\in\Aut(T_{d})\mid\forall i\in I:\ \exists j\in J:\ \forall v\in P_{i}:\ \sigma(g,v)\in H_{j}\}
\end{displaymath}
denote the set of automorphisms of $T_{d}$ whose local permutations at the vertices of a given element of $\calP$ all come from the same element of $\calH$.

\begin{proposition}\label{prop:localization_uf}
Let $\smash{F\le F'\le\widehat{F}\le\Sym(\Omega)}$ and $F(p)\le F$ a $p$-Sylow subgroup such that $F(p)_{\omega}\le F_{\omega}$ is a $p$-Sylow subgroup for all $\omega\in\Omega$. Set $S:=\mathrm{U}(F(p))_{b}$. Then
\begin{align*}
  \mathrm{Comm}_{\mathrm{G}(F,F')}(S)&=\langle\mathrm{U}(\{\id\}),\mathrm{Comm}_{\mathrm{G}(F,F')_{b}}(S)\rangle \\
  &\ge\langle \mathrm{G}(F(p),F'),\{\Gamma_{V/L}(N_{F}(F(p))/F(p))\mid L\le S \text{ open}\}\rangle.
\end{align*}
\end{proposition}

\begin{proof}
By Proposition \ref{prop:local_sylow_uf}, the group $S$ is a local $p$-Sylow subgroup of $\mathrm{U}(F)$ and hence of $\mathrm{G}(F,F')$. We first show that $\mathrm{G}(F,F')_{(p)}$ contains $\mathrm{U}(\{\id\})$. Indeed, given $g\in\mathrm{U}(\{\id\})$ we have $gSg^{-1}=\mathrm{U}(F(p))_{g(b)}$. Thus $S\cap gSg^{-1}=\mathrm{U}(F(p))_{(b,g(b))}$ which has finite index in both $S=\mathrm{U}(F)_{b}$ and $gSg^{-1}=\mathrm{U}(F(p))_{g(b)}$ by the orbit-stabilizer theorem. Since $\mathrm{U}(\{\id\})$ acts vertex-transitively on $T_{d}$ we conclude
\begin{displaymath}
 \mathrm{Comm}_{\mathrm{G}(F,F')}(S)=\langle\mathrm{U}(\{\id\}),\mathrm{Comm}_{\mathrm{G}(F,F')_{b}}(S)\rangle.
\end{displaymath}
Now, the vertex stabilizer $\mathrm{G}(F,F')_{b}$ is residually discrete by Proposition \ref{prop:gf_res_disc}. Hence, by Lemma \ref{lem:comm_normalizer}, the commensurator $\mathrm{Comm}_{\mathrm{G}(F,F')_{b}}(S)$ is the union of the normalizers in $\mathrm{G}(F,F')_{b}$ of open subgroups of $S=\mathrm{U}(F(p))_{b}$. For example, we may consider $L_{n}:=\mathrm{U}(F(p))_{B(b,n)}\le_{o}S$ for every $n\in\bbN$. The normalizer of $L_{n}$ in $\mathrm{G}(F,F')_{b}$ contains those elements of $\mathrm{G}(F(p),F')_{b}$ all of whose singularities are contained in $B(b,n)$. Taking the union over all $n\in\bbN$ and using vertex-transitivity of $\mathrm{G}(F(p),F')$ in the sense that $\mathrm{G}(F(p),F')=\langle\mathrm{G}(F(p),F')_{b},\mathrm{U}(\{\id\})\rangle$ we conclude that $\mathrm{Comm}_{\mathrm{G}(F.F')}(S)$ contains $\mathrm{G}(F(p),F')$ as a topological subgroup. Alternatively, use \cite[Lemma 3.2]{Bou16}.

As to $\Gamma_{V/L}(N_{F}(F(p)))$, note that for all $g, s\in\Aut(T_{d})$ and $v\in V$ we have
\begin{align*}
 \sigma(gsg^{-1},v)&=\sigma(g,sg^{-1}v)\sigma(s,g^{-1}v)\sigma(g^{-1},v) \\
 &=\sigma(g,sg^{-1}v)\sigma(s,g^{-1}v)\sigma(g,g^{-1}v)^{-1}.
\end{align*}
Hence if $g\in\Gamma_{V/L}(N_{F}(F(p))/F(p))$, i.e. the coset $\sigma(g,v)F(p)\subseteq N_{F}(F(p))$ is constant on $L$-orbits, then $gLg^{-1}\subseteq\mathrm{U}(F(p))$ whence $g\in\mathrm{Comm}_{\mathrm{G}(F,F')}(S)$.
\end{proof}

\begin{remark}
Whereas the next result provides conditions on $F\le\Sym(\Omega)$ which ensure $\mathrm{U}(F)_{(p)}=\mathrm{G}(F(p),F)$ and we have $\mathrm{U}(F)_{(p)}=\mathrm{U}(F)$ for semiregular $F$ by Proposition \ref{prop:uf_properties}, it may happen that $\mathrm{G}(F(p),F)\lneq\mathrm{U}(F)_{(p)}\lneq\mathrm{U}(F)$. Indeed, if for every $\omega\in\Omega$ there is an element $a_{\omega}\in F_{\omega}$ such that for all $\lambda\in\Omega$ we have $F(p)_{\lambda}\cap a_{\omega}F(p)_{\lambda}a_{\omega}^{-1}=\{\id\}$ then there is an element $g\in\mathrm{U}(F)_{B(b,1)}$ such that for $S:=\mathrm{U}(F(p))_{B(b,1)}$ we have $S\cap gSg^{-1}=\{\id\}$ and therefore $g\notin\mathrm{U}(F)_{(p)}$: Choose the local permutation of $g$ at $v\in V(T_{d})$ to be $a_{\omega}$ whenever $d(v,b)=d(v,b_{\omega})+1$. If in addition $\mathrm{N}_{F}(F(p))\gneq F(p)$ then the assertion holds by virtue of Proposition \ref{prop:localization_uf}. For instance, these assumptions are satisfied for $F=S_{6}$ and $p=3$.
\end{remark}

\begin{theorem}\label{thm:localization_uf}
Let $\smash{F\le F'\le\widehat{F}\le\Sym(\Omega)}$ and $F(p)\le F$ a $p$-Sylow subgroup of~$F$. Assume that we have $F\backslash\Omega=F(p)\backslash\Omega$ and $N_{F_{\omega}'}(F(p)_{\omega})=F(p)_{\omega}$ for all $\omega\in\Omega$. Then~$\mathrm{G}(F,F')_{(p)}=\mathrm{G}(F(p),F')$.
\end{theorem}

If $F$ does not fix a point of $\Omega$ and $F\backslash\Omega=F(p)\backslash\Omega$ then $p$ divides $|\Omega|$. By Proposition \ref{prop:local_sylow_orbits} the same assumption implies that the point stabilizers in $F(p)$ are $p$-Sylow subgroups of the respective point stabilizers in $F$. In the case $F=F'$, the theorem asks that these be self-normalizing.

\begin{proof}
(Theorem \ref{thm:localization_uf}). By Proposition \ref{prop:local_sylow_uf} and Proposition \ref{prop:localization_uf} it suffices to show that $\mathrm{Comm}_{\mathrm{G}(F,F')_{b}}(\mathrm{U}(F(p))_{b})\!=\!\mathrm{G}(F(p),F')_{b}$. By Proposition \ref{prop:localization_uf}, the group $\mathrm{G}(F(p),F')_{b}$ is a subgroup of said commensurator.

Now suppose $g\in\mathrm{Comm}_{\mathrm{G}(F,F')_{b}}(\mathrm{U}(F(p))_{b})\le\mathrm{G}(F,F')_{b}$. Given that $\mathrm{G}(F,F')_{b}$ is residually discrete by Proposition \ref{prop:gf_res_disc}, the element $g$ normalizes an open subgroup $L\le\mathrm{U}(F(p))_{b}$ by virtue of Lemma \ref{lem:comm_normalizer}. If $g$ has only finitely many local permutations in $F'\backslash F(p)$ then $g\in G(F(p),F')_{b}$. Otherwise, the above implies that there is $n\in\bbN$ such that $g\mathrm{U}(F(p))_{B(b,n)}g^{-1}\subseteq L\subseteq\mathrm{U}(F(p))_{b}$ and $g$ has a local permutation in $F'\backslash F(p)$ on $S(b,n)$. Then construct $h\in\mathrm{G}(F(p),F')$ with local permutations in $F(p)$ on spheres of radius at least $n$ and such that $h^{-1}g$ fixes $B(b,n)$ pointwise as follows: Set $h|_{B(b,n-1)}:=g$ and use the assumption $F'\backslash\Omega=F\backslash\Omega=F(p)\backslash\Omega$ to extend $h$ to all $T_{d}$ using $F(p)$ only. Then $h^{-1}g$ has a local permutation in $F_{\omega}'\backslash F(p)_{\omega}$ for some $\omega\in\Omega$ on $S(b,n)$ and $(h^{-1}g)\mathrm{U}(F(p))_{B(b,n)}(h^{-1}g)^{-1}\subseteq L\subseteq\mathrm{U}(F(p))_{b}$. However, this contradicts the assumption $N_{F_{\omega}'}(F(p)_{\omega})=F(p)_{\omega}$ for all $\omega\in\Omega$.
\end{proof}

Theorem \ref{thm:localization_uf} can be used to determine the $p$-localization of Lederle's coloured Neretin group $\mathrm{N}(F)$ under similar assumptions.

\begin{theorem}\label{thm:localization_nf}
Let $F\le\Sym(\Omega)$ and $F(p)\le F$ a $p$-Sylow subgroup. If $F\backslash\Omega=F(p)\backslash\Omega$ and $N_{\widehat{F}_{\omega}}(F(p)_{\omega})=F(p)_{\omega}$ for all $\omega\in\Omega$ then $\mathrm{N}(F)_{(p)}=\mathrm{N}(F(p))$.
\end{theorem}

\begin{proof}
By Proposition \ref{prop:local_sylow_uf}, the group $S:=\mathrm{U}(F(p))_{b}$ is a local Sylow subgroup of $\mathrm{N}(F)$. Furthermore, by \cite[Proposition 2.24]{Led17}, we have $\mathrm{N}(F(p))\le\mathrm{Comm}_{\mathrm{N}(F)}(S)$. Now, let $g\in\mathrm{Comm}_{\mathrm{N}(F)}(S)$ and let $g:T_{d}\backslash T\to T_{d}\backslash T'$ be a representative of $g$ as an $\mathrm{U}(F)$-honest almost automorphism. Given that $F\backslash\Omega=F(p)\backslash\Omega$ there is a $\mathrm{U}(F(p))$-honest almost automorphism $h\in\mathrm{N}(F(p))\le\mathrm{Comm}_{\mathrm{N}(F)}(S)$ with representative $h:T_{d}\backslash T'\to T_{d}\backslash T$ such that $hg:T_{d}\backslash T\to T_{d}\backslash T$ fixes the leaves of $T$ and therefore extends to an autormorphism of $T_{d}$ fixing $T$. Furthermore, on each connected component of $T_{d}\backslash T$, the automorphism $hg\in\mathrm{N}(F)\cap\Aut(T_{d})$ coincides with an element of $\mathrm{U}(F)$. Hence, using Proposition \ref{prop:uf_tits}, we have $hg\in\mathrm{U}(F)$ and therefore
\begin{displaymath}
 hg\in\mathrm{Comm}_{\mathrm{N}(F)\cap\Aut(T_{d})}(S)=\mathrm{Comm}_{\mathrm{G}(F)}(S)=\mathrm{G}(F)_{(p)}=\mathrm{G}(F(p))\le\mathrm{N}(F(p)).
\end{displaymath}
by Theorem \ref{thm:localization_uf}. Given that $h\in\mathrm{N}(F(p))$ we conclude $g\in\mathrm{N}(F(p))$ as required.
\end{proof}

Proposition \ref{prop:localization_uf} suggests that Theorem \ref{thm:localization_uf} might hold as soon as $F(p)$ is self-normalizing in $F'$. This is not the case as the following remark shows.

\begin{remark}\label{rem:localization_uf}
Theorem \ref{thm:localization_uf} does not hold when the condition $N_{F_{\omega}'}(F(p)_{\omega})=F(p)_{\omega}$ for all $\omega\in\Omega$ is replaced with $N_{F'}(F(p))=F(p)$: There are transitive, non-regular permutation groups $F\le\Sym(\Omega)$ and primes $p$ such that $F\backslash\Omega=F(p)\backslash\Omega$ and $N_{F}(F(p))=F(p)$ for which $F(p)$ is regular. In particular, $N_{F_{\omega}}(F(p)_{\omega})\gneq F(p)_{\omega}$. In this case, $\mathrm{U}(F(p))_{b}$ is a local $p$-Sylow subgroup of $\mathrm{U}(F)$ by Proposition \ref{prop:local_sylow_orbits}. However, $\mathrm{U}(F(p))_{b}\cong F(p)$ is finite and hence $\mathrm{U}(F)_{(p)}=\mathrm{U}(F)\gneq\mathrm{G}(F(p),F)$.

A small example of this situation is a certain $F\cong S_{4}\le S_{8}$ and the prime $p=2$, namely put $F:=\langle(123)(456),(14)(25)(37)(68)\rangle$. Here, $F(2)$ is regular and self-normalizing in $F$ of order eight.
\end{remark}

\bibliographystyle{amsalpha}
\bibliography{localization}

\end{document}